\documentclass[11pt,a4paper,reqno]{amsart}
\usepackage[latin1]{inputenc}
\usepackage[T1]{fontenc}
\usepackage{lmodern}
 \usepackage[english]{babel,minitoc}
\usepackage{hyperref}
\usepackage{amsfonts}
\usepackage{amsmath}

\numberwithin{equation}{section}\theoremstyle{plain}
\newtheorem{theorem}{Theorem}[section]

\newtheorem{lemma}[theorem]{Lemma}

\address{\begin{center}{\small Department of Mathematics and Computer Sciences, Faculty of Sciences,\\
Equipe d'Analyse Harmonique et Probabilit\'{e}s, University Moulay Isma\"{\i}l,\\
BP 11201 Zitoune, Meknes, Morocco}
\end{center}}

\begin{document}

\title[A heat kernel version of Miyachi's Theorem for the Laguerre hypergroup]
{A heat kernel version of Miyachi's Theorem for the Laguerre hypergroup}

\author[S.Fahlaoui and M.El kassimi]{ S.Fahlaoui \quad and \quad M.El kassimi}

\address{Sa\"{\i}d Fahlaoui}  \email{s.fahlaoui@fs.umi.ac.ma}

\address{Mohammed El kassimi} \email{m.elkassimi@edu.umi.ac.ma}

\maketitle
\begin{abstract}
Let $\mathbb{K}=[0,+\infty[\times\mathbb{R}$ the Laguerre Hypergroup.
In this paper, we are going to formulate and  prove  an analogue  of Miyachi's uncertainty principle  for the  Laguerre-Hypergroup Fourier transform. Our version will be in terms of the heat kernel associated to the radial part of the sub-Laplacian on the Heisenberg group.
\end{abstract}

{\it keywords} Miyachi's theorem, Laguerre-Hypergroup, Fourier Laguerre transform\\

2010 MATHEMATICS Subject Classification 42A38; 42B10; 43A32

\section{Introduction}

   A wonderful aspect of  quantum physics is that, we cannot measure  the position and momentum of a particle simultaneously with high precision.  The  mathematical formulation of this rule is that, we cannot at the same time localize the value of a function and its Fourier transform. There are many formulations of this idea, previously developed by Heisenberg in 1927 \cite{Heis}. Later, in 1933 Hardy have obtained  a new formulation of this principle \cite{Har}. After, in 1997 Miyachi \cite {Miya} proved the following theorem for the real line:

\begin{theorem}
Let $f$ be an integrable function on $\mathbb{R}$ such that
$$e^{ax^2}f \in L^1(\mathbb{R}) + L^{\infty}(\mathbb{R}),$$

 Further assume that
$$\int_\mathbb{R}\log^+(\frac{|e^{b\lambda^2}\widehat{f}(\lambda)|}{\delta})d\lambda<\infty,$$

for some   numbers $a, b, \delta>0$.\\
~~$(i)$ if $ab>\frac{1}{4}$, then $f=0$ a.e.\\
~$(ii)$If $ab = \frac{1}{4}$, then f is a constant multiple of the gaussian $e^{-ax^2}$.

\end{theorem}

  For the Laguerre-Hypergroup Fourier transform in \cite{Huang} H.Jizheng and L.Heping proved the Hardy's theorem, and in \cite{Huang2} they demonstrated Beurling's theorem. In this paper we are going to give a version of Miyachi's theorem for the Laguerre-Hypergroup.

\section{Harmonic Analysis for Laguerre Hypergroup}

We consider the following partial differential operator\\
$$\left\{
  \begin{array}{ll}
    D=\frac{\partial}{\partial x},  \\
    \mathcal{L}=\frac{\partial^2}{\partial x^2}+\frac{2\alpha+1}{x}\frac{\partial}{\partial x}+x^2\frac{\partial^2}{\partial t^2}, \\
    (x,t)\in[0,+\infty[ \times \mathbb{R}~ and~ \alpha\in [0,+\infty[.
  \end{array}
\right.$$
For $\alpha=n-1$, $n\in\mathbb{N}^{*}$, the operator $\mathcal{L}$ is the radial part of the sub-Laplacian on the Heisenberg group $\mathbb{H}_{n}$.

For $(\lambda,m)\in \mathbb{R}\times\mathbb{N}$, the initial value problem

 $$\left\{
         \begin{array}{ll}
         Du=i\lambda u, \\
         \mathcal{L} u=-4|\lambda|(m+\frac{\alpha+1}{2})u,  \\
         u(0,0)=1, \frac{\partial u}{\partial x}(0,t)=0\quad for~all~~t\in \mathbb{R},
          \end{array}
 \right.$$\\
has a unique solution $\phi_{\lambda,m}$ given by
\begin{equation}\label{func1}
\phi_{\lambda,m}(x,t)=e^{i\lambda t}\mathcal{L}^{(\alpha)}_{m}(|\lambda|x^2),\quad (x,t)\in \mathbb{K},
\end{equation}\\
 where $\mathcal{L}^{(\alpha)}_{m}$ is the Laguerre function on $\mathbb{R}_{+}$ defined  by

$$\mathcal{L}^{(\alpha)}_{m}(x)=e^{-\frac{x}{2}}L_{m}^{(\alpha)}(x)/L_{m}^{(\alpha)}(0)$$\\
and $L^{(\alpha)}_{m}$ is the Laguerre polynomial of degree $m$ and order $\alpha$ defined in terms of the generating function by(see \cite{Assal}):\\
\begin{equation}\label{ass1}
 \sum_{m=0}^{\infty}s^{m}L_{m}^{\alpha}(x)=\frac{1}{(1-s)^{\alpha+1}}exp(-\frac{xs}{1-s})
\end{equation}

Set $$\varphi_{m}^{\alpha}(x)=e^{-\frac{x^2}{2}}L_{m}^{(\alpha)}(x^2)$$
\begin{lemma}\label{lem3}
For any $\lambda\neq0$, the system\\
$$\{(\frac{2|\lambda|^{\alpha+1}m!}{\Gamma(m+\alpha+1)})^{\frac{1}{2}}\varphi^{(\alpha)}_{m}(\sqrt{|\lambda|}x);m\in\mathbb{N}\}$$\\
forms an orthonormal  basis of space $L^{2}([0,+\infty[,x^{2\alpha+1}dx).$\\
\end{lemma}
For $(\lambda,m)\in \mathbb{R}\times\mathbb{N}$, we put
$$\psi_{\lambda}(x,t)=\frac{m!\Gamma(\alpha+1)}{\Gamma(m+\alpha+1)}e^{i\lambda t}\varphi^{(\alpha)}_{m}(\sqrt{|\lambda|}x).$$
\\

 Let $\alpha\geq0$ be a fixed number, $\mathbb{K}=[0,+\infty[\times\mathbb{R}$ and $m_{\alpha}$ the weighted Lebesgue measure on $\mathbb{K}$, given by
\\
$$dm_{\alpha}(x,t)=\frac{x^{2\alpha+1}dxdt}{\pi\Gamma(\alpha+1)},\quad \alpha\geq0.$$
\\
For $(x,t)\in \mathbb{K}$, the generalized translation operator $T^{(\alpha)}_{(x,t)}$ is defined by\\
$$T^{(\alpha)}_{(x,t)}f(y,s)=
                     \left\{
                               \begin{array}{ll}
                                 \frac{1}{2\pi}\int_{0}^{2\pi}f(\sqrt{x^2+y^2+2xycos\theta},s+t+xysin\theta)d\theta,\quad if\quad\alpha=0 \\\\
                                 \frac{\alpha}{\pi}\int_{0}^{2\pi}\int_{0}^{1}f(\sqrt{x^2+y^2+2xycos\theta},s+t+xysin\theta)r(1-r^2)^{\alpha-1}dr d\theta,\quad if \quad\alpha>0
                               \end{array}
                             \right.$$\\
\\
Let $M_{b}(\mathbb{K})$ denote the space of bounded Radon measures on $\mathbb{K}$.\\
The convolution on $M_{b}(\mathbb{K})$ is defined by\\
$$(\mu\ast\nu)(f)=\int_{\mathbb{K}\times\mathbb{K}}T^{\alpha}(x,t)f(y,s)d\mu(x,t)d\nu(y,s).$$\\

It is seen that $\mu\ast\nu=\nu\ast\mu$. If $f,g\in L^{1}(\mathbb{K})$ and $\mu=fm_{\alpha}$, $\nu=gm_{\alpha}$, then $\mu\ast\nu=(f\ast g)m_{\alpha}$, where $f\ast g$ is the convolution of functions $f$ and $g$, defined by
\\
$$(f\ast g)(x,t)=\int_{\mathbb{K}}T^{\alpha}_{x,t}f(y,s)g(y,-s)dm_{\alpha}(y,s).$$
\\
For every $1\leq p<\infty$, we denote by $L_{p}(\mathbb{K})=L_{p}(\mathbb{K},dm_{\alpha})$ the space of complex-valued functions $f$, measurable on $\mathbb{K}$ such that\\
$$ \|f\|_{L_{p}(\mathbb{K})}=(\int_{\mathbb{K}}|f(x,t)|^{p}dm_{\alpha}(x,t))^{\frac{1}{p}},\quad if ~ p\in[1,+\infty[.$$\\
and
$$\|f\|_{L^{\infty}}(\mathbb{K})=\displaystyle ess~sup_{(x,t)\in \mathbb{K}}|f(x,t)|.$$

The following proposition summarizes some basic properties of functions $\psi_{\lambda,m}$ (see\cite{Korta}).
\begin{lemma}\label{lem0}
The functions $\psi_{\lambda,m}$ satisfy that
\begin{itemize}
  \item $\|\psi_{\lambda,m}\|_{L^{\infty}}=\psi_{\lambda,m}(0,0)$,
  \item $T^{(\alpha)}_{(x,t)}\psi_{\lambda,m}(y,s)=\psi_{\lambda,m}(x,t)\psi_{\lambda,m}(y,s)$,
  \item $\mathcal{L}\psi_{\lambda,m}=2|\lambda|(2m+\alpha+1)\psi_{\lambda,m}$.
\end{itemize}

\end{lemma}
\textbf{Fourier Laguerre transform}\\
Let $f\in L^{1}(\mathbb{K})$, the generalized Fourier transform of $f$ is defined by
$$\widehat{f}(\lambda,m)=\int_{\mathbb{K}}f(x,t)\psi_{(-\lambda,m)}dm_{\alpha}(x,t).$$
We note that
$$\widehat{f}(\lambda,m)=\frac{m!}{\pi\Gamma(m+\alpha+1)}\int_{0}^{+\infty}f^{\lambda}(x)\varphi^{\alpha}_{m}(\sqrt{|\lambda|}x)x^{2\alpha+1}dx$$
 where
 \begin{equation}\label{equ1}
f^{\lambda}(x)=\int_{-\infty}^{+\infty}f(x,t)e^{-i\lambda t} dt,
\end{equation}\\
is the classical Fourier transform of $f(x,t)$ in the variable $t$.\\
Let $d\gamma_{\alpha}$ be the positive measure defined on $\mathbb{R}\times\mathbb{N}$ by
\\
$$\int_{\mathbb{R}\times\mathbb{N}}g(\lambda,m)d\gamma_{\alpha}(\lambda,m)=\sum_{m=0}^{\infty}\frac{\Gamma(m+\alpha+1)}{m!\Gamma(\alpha+1}\int_{\mathbb{R}}
g(\lambda,m)|\lambda|^{\alpha+1}d\lambda.$$
\\
Write $L^{p}(\widehat{\mathbb{K}})$ instead of $L^{p}(\mathbb{R}\times\mathbb{N},\gamma_{\alpha})$.\\
We have the following Plancherel formula:
\\
$$\|f\|_{L^{2}(\mathbb{K})}=\|\widehat{f}\|_{L^{2}(\widehat{\mathbb{K}})}; \quad for ~ f\in L^{1}(\mathbb{K})\cap L^{2}(\mathbb{K}).$$
\\
We also have the inverse formula of the generalized Fourier transform:
\\
$$f(x,t)=\int_{\mathbb{R}\times\mathbb{N}}\widehat{f}(\lambda,m)\psi_{(\lambda,m)}(x,t)d\gamma_{\alpha}(\lambda,m), $$
\\
provided $\widehat{f}\in L^{1}(\mathbb{\widehat{K}})$.\\

\textbf{Heat kernel}\\
Let $\{H^{s}: s>0\}=\{e^{-s\mathcal{L}}: s>0\}$ be the heat semigroup generated by $\mathcal{L}$.\\ There is a unique smooth function
$h((x,t),s)=h_{s}(x,t)$ on $\mathbb{K}\times]0,+\infty[$, such that
$$ H^{s}f(x,t)=f*h_{s}(x,t).$$
$h_{s}$ is called the heat kernel associated to $\mathcal{L}$.\\
By the definition of the generalized Fourier transform and lemma \ref{lem0}, it is known that
\begin{lemma}\label{ass2}
\begin{eqnarray*}
\widehat{\mathcal{L}f}&=&2|\lambda|(2m+\alpha+1)\widehat{f}(\lambda,m),\\
(\widehat{f\ast g})(\lambda,m)&=&\widehat{f}(\lambda,m)\widehat{g}(\lambda,m).
\end{eqnarray*}
Therefore
\begin{eqnarray*}
\widehat{h_{s}}(\lambda,m) &=&e^{-|\lambda|(2m+\alpha+1)s}\\
h_{s_{1}}\ast h_{s_{2}}&=&h_{s_{1}+s_{2}},\\
h_{s}(x,t)&=&s^{-(\alpha+2)}h_{1}(\frac{x}{\sqrt{s}},\frac{t}{s}).
\end{eqnarray*}

\end{lemma}
Although the heat kernel $h_{s}(x,t)$ is not explicitly known, we  have   an explicit expression of   $h_{s}$ in terms of Euclidean
Fourier transform with respect to the variable $t$ \cite{Huang}.\\

\begin{lemma}\label{ass3}
The Fourier transform of the Laguerre heat kernel is given by,
$$
h_{s}(x,t)=\int_{\mathbb{R}}(\frac{\lambda}{2\sinh(2\lambda s)})^{\alpha+1}e^{-\frac{1}{2}\lambda Coth(2\lambda s)x^2}e^{i\lambda t} d\lambda.
$$
\end{lemma}

The pointwise estimate of the heat kernel $h_{s}(x, t)$ can be derived from its fourier transform expression, we have the next lemma, \\

\begin{lemma}\label{estimat}
There exists $A>0$ such that
$$ 0<h_{s}(x,t)\leq Cs^{-\alpha+2}e^{-\frac{A}{s}(|x|^{2}+|t|)}.$$

\end{lemma}
\begin{proof}
 For the demonstration we can see \cite{Huang1}.
\end{proof}
 Now we turn  to the Hankel transform. For $z\in\mathbb{C}$, the Bessel function of first kind and order $\alpha$ is defined by

$$J_{\alpha}(z)=\sum_{k=0}^{\infty}\frac{(-1)^{k}2^{-\alpha-2k}z^{\alpha+2k}}{\Gamma(k+1)\Gamma(k+\alpha+1)}=
\frac{2^{-\alpha}z^{\alpha}}{\sqrt{\pi}\Gamma(\alpha+\frac{1}{2})}\int_{-1}^{1}e^{iz}(1-s^2)^{\alpha-\frac{1}{2}}ds.$$

Suppose that $\alpha\geq0$.\\
The Hankel transform of order $\alpha$ of $f\in L^{1}([0,+\infty[,x^{2\alpha+1}dx)$ is defined by
$$ (\mathcal{H}_{\alpha}f)(y)=\int_{0}^{+\infty}f(x) \frac{J_{\alpha(xy)}}{(xy)^{\alpha}}x^{2\alpha+1}dx.$$

 The functions $\varphi_{m}(x)$ defined in \ref{func1} are the eigenfunctions of the Hankel transform, that is,
$$(\mathcal{H}_{\alpha}\varphi_{m})(x)=(-1)^{m}\varphi_{m}(x).$$

\section{Miyachi's theorem}
In the proof of Miyachi's uncertainty principle, in the most cases of Fourier transform we use the following approach, first,   the transform is an entire function, in the second  we prove that, this transform verified the conditions of Miyachi's lemma, after we conclude the result. But for the Laguerre-Hypergroup Fourier transform is not a holomorphic function, so for that we are going to use a new trick that we write the Laguerre-hypergroup Fourier transform of a function satisfying the condition \ref{equa10} as an infinite sum of elements of the basis in the lemma \ref{lem3}, from this we will try to find the conditions of the Miyachi's lemma \ref{lem5}.\\

To prove our main result, we need the following lemmas,\\

We begin by the following lemma  proved by A. Miyachi in \cite{Miya},
\begin{lemma}\label{miya0}
Let $f$ be an entire function on $\mathbb{C}$ and suppose there exist constants, $A,B>0 $ such that :
$$|f(z)|\leq A e^{B(Re(z))^2} ~~for~ all~ z\in \mathbb{C} $$
Also suppose
$$\int_{-\infty}^{+\infty}log^{+}|f(t)|dt<+\infty,$$
Where $log^{+}(x)=log(x)$ if $x>1$ and $log^{+}(x)=0$ if $x<1$.\\
Then $f$ is a constant function.
\end{lemma}

\begin{lemma}\label{lem5}
Let $h$ be an entire function on $\mathbb{C}$ and $\alpha \geq \frac{-1}{2}$, such that
$$\forall z\in \mathbb{C}, |h(z)|\leq A e^{B(Re(z))^2};$$
  for some constants $A,B>0$, and

$$\int_{-\infty}^{+\infty}log^{+}|h(t)||t|^{2\alpha+1}dt<+\infty,$$
Then $h$ is a constant function.
\end{lemma}
\begin{proof}
  We have
  \begin{equation}
    \int_{-\infty}^{+\infty}log^{+}|h(t)|dt=\int_{-1}^{+1}log^{+}|h(t)|dt+\int_{|t|>1}log^{+}|h(t)|dt,
  \end{equation}
  and
   \begin{equation}\label{miya1}
    \int_{|t|>1}log^{+}|h(t)|dt\leq\int_{|t|>1}log^{+}|h(t)|t|^{2\alpha+1}dt<+\infty,
   \end{equation}

  $h$ is an entire, particularly is continuous, then
   \begin{equation}\label{miya2}
  \int_{-1}^{+1}log^{+}|h(t)|dt<+\infty,
   \end{equation}
  we deduce from \ref{miya1} and \ref{miya2} that
  $$ \int_{-\infty}^{+\infty}log^{+}|h(t)|dt<+\infty,$$
  the  lemma \ref{miya0} finishes the proof.
\end{proof}
We  put $$S_a= \{ \lambda \in \mathbb{C}/ |Im \lambda|<4aA\},$$
where $A$ is the constant in heat kernel estimate \ref{estimat}.\\
The purpose of the following lemma is to prove that, if we have a function $f$ satisfied the condition
\begin{equation*}
h^{-1}_{\frac{1}{4a}}f\in L^{1}(\mathbb{K})+L^{\infty}(\mathbb{K}),
\end{equation*}
so $f^{\lambda}$  is in the space $L^{2}([0,+\infty),x^{2\alpha+1}dx)$.
\begin{lemma}\label{lem1}
Let $f$ be a measurable function and $a$ is a positive constant, such that,
$$h^{-1}_{\frac{1}{4a}}f\in L^{1}(\mathbb{K})+L^{\infty}(\mathbb{K})$$
for $\lambda\in S_a$, we have
$$|f^{\lambda}(x)|\leq C e^{-4aAx^2}.$$
where $C$ is a positive constant.
\end{lemma}
\begin{proof}
\quad Let $h^{-1}_{\frac{1}{4a}}f\in L^{1}(\mathbb{K})+L^{\infty}(\mathbb{K})$,\\
Then, there are two functions $u\in L^{1}(\mathbb{K})$ and $v\in L^{\infty}(\mathbb{K})$, such that:
$$h^{-1}_{\frac{1}{4a}}(x,t)f(x,t)=u(x,t)+v(x,t),$$
So
$$f(x,t)=h_{\frac{1}{4a}}(x,t)u(x,t)+h_{\frac{1}{4a}}(x,t)v(x,t),$$
by \ref{equ1}, we have
$$f^{\lambda}(x)=\int_{-\infty}^{+\infty}h_{\frac{1}{4a}}(x,t)u(x,t)e^{-i\lambda t}dt + \int_{-\infty}^{+\infty}h_{\frac{1}{4a}}(x,t)v(x,t)e^{-i\lambda t}dt$$

for $\lambda\in \mathbb{C}$ with $(\lambda=\xi+i\eta)$, we have \\
$\displaystyle|\int_{-\infty}^{+\infty}h_{\frac{1}{4a}}(x,t)u(x,t)e^{-i\lambda t}dt|$
\begin{eqnarray*}
&\leq& C\int_{-\infty}^{+\infty}e^{-4aA(x^2+|t|)}|u(x,t)|e^{\eta t}dt\\
&\leq& C e^{-4aAx^2}\left(\int_{-\infty}^{0}e^{4aAt+\eta t}|u(x,t)|dt+\int_{0}^{+\infty}e^{-4aAt+\eta t}|u(x,t)|dt\right)\\
&=& C e^{-4aAx^2}\left(\int_{-\infty}^{0}e^{(4aA+\eta) t}|u(x,t)|dt+\int_{0}^{+\infty}e^{(-4aA+\eta )t}|u(x,t)|dt\right)
\end{eqnarray*}
Then   $\displaystyle|\int_{-\infty}^{+\infty}h_{\frac{1}{4a}}u(x,t)e^{-i\lambda t}dt|<+\infty$ if and only if
            $$-4aA<\eta<4aA $$
So, for $\lambda \in S_a $ we have

\begin{equation}\label{aqu2}
|\int_{-\infty}^{+\infty}h_{\frac{1}{4a}}(x,t)u(x,t)e^{-i\lambda t}dt|\leq Ce^{-4aAx^2},
\end{equation}
and we have\\
$\displaystyle|\int_{-\infty}^{+\infty}h_{\frac{1}{4a}}(x,t)v(x,t)e^{-i\lambda t}dt|$
\begin{eqnarray*}
&\leq& C\int_{-\infty}^{+\infty}e^{-4aA(x^2+|t|)}|v(x,t)|e^{\eta t}dt\\
&\leq&C e^{-4aAx^2}\int_{-\infty}^{+\infty}e^{-4aA|t|}|v(x,t)|e^{\eta t}dt\\
&\leq& C \|v\|_{\infty}e^{-4aAx^2}\left(\int_{-\infty}^{0}e^{4aAt+\eta t}dt+\int_{0}^{+\infty}e^{-4aAt+\eta t}dt\right)
\end{eqnarray*}
 Then $\displaystyle|\int_{-\infty}^{+\infty}h_{\frac{1}{4a}}v(x,t)e^{-i\lambda t}dt|<+\infty$ if and only if~~
            $-4aA<\eta<4aA $,\\
and therefore if $\lambda\in S_{a}$, we have
\begin{equation}\label{aqu3}
|\int_{-\infty}^{+\infty}h_{\frac{1}{4a}}(x,t)v(x,t)e^{-i\lambda t}dt|\leq Ce^{-4aAx^2}
\end{equation}
by the inequalities \ref{aqu2} and \ref{aqu3} we deduce that
$$|f^{\lambda}(x)|\leq C e^{-4aAx^2}.$$
\end{proof}
As in the paper \cite{Huang}, we have,
when the function $f$ satisfies the condition \ref{equa10}, we have the following estimation for the Hankel transform of the function $f$,
\begin{lemma}\label{lem2}
Let $f$ be a measurable function and $a$ is a positive constant, such that
$$h^{-1}_{\frac{1}{4a}}f\in L^{1}(\mathbb{K})+L^{\infty}(\mathbb{K}),$$
for $\lambda\in S_a$, we have,
$$|\mathcal{H}_{\alpha}(f^{\lambda})(z)|\leq Ce^{\frac{(Im(z))^{2}}{4aA}}~~ for~all~ z\in\mathbb{C}. $$
where $C$ is a positive constante.
\end{lemma}

Our main results is the following theorem:

\begin{theorem}\label{miyachi-theo}
Let $f$ be a measurable function on $\mathbb{K}$  such that
\begin{equation}\label{equa10}
h^{-1}_{\frac{1}{4a}}f\in L^{1}(\mathbb{K})+L^{\infty}(\mathbb{K}),
\end{equation}
and
$$\int_{\mathbb{R}}\log^{+}(\frac{|\mathcal{H}_{\alpha}(f^{\lambda})(y)e^{by^2}|}{\delta})|y|^{2\alpha+1}dy<+\infty.$$
where $\delta$ is positive constant.
where $a,b>0$. \\
 If $ab>\frac{1}{4}$ then,  $f(x,t)=0$ a.e.

\end{theorem}
\begin{proof}

\quad First, let $g$ a function defined by $g(x)=e^{\frac{x^2}{4a}}\mathcal{H}_{\alpha}(f^{\lambda})(x)$\\
$\ast$ If $ab>\frac{1}{4}$ \\
we have\clearpage
$$\displaystyle\int_{-\infty}^{+\infty} log^{+}|g(x)||x|^{2\alpha+1}dx\leq$$
\begin{eqnarray*}&\leq& \int_{-\infty}^{+\infty}log^{+}(|\mathcal{H}_{\alpha}(f^{\lambda})(x)e^{bx^2}|)|x|^{2\alpha+1}dx+\int_{-\infty}^{+\infty}e^{(\frac{1}{4a}-b)x^2}|x|^{2\alpha+1}dx\\
&<& +\infty
\end{eqnarray*}
 As $z\mapsto\mathcal{H}_{\alpha}(f^{\lambda})(z)$ is an entire  function, so $z\mapsto g(z)$ is also entire. For $\lambda\in S_a$ applying the lemma \ref{lem5}, we get  $g$  is a constant.\\
Thus, for $ab>\frac{1}{4}$,
there is a constant $C$ such that
 $$g(x)=C\Leftrightarrow e^{\frac{1}{4a}x^2}\mathcal{H}_{\alpha}(f^{\lambda})(x)=C$$
 then  $$\mathcal{H}_{\alpha}(f^{\lambda})(x)=Ce^{-\frac{1}{4a}x^2}$$
 but in this case, the relation
 \begin{equation}\label{equa5}
   \int_{-\infty}^{+\infty}log^{+}(\frac{|\mathcal{H}_{\alpha}(f^{\lambda})(x)e^{bx^2}|}{\delta})|x|^{2\alpha+1}dx<+\infty
 \end{equation}

 holds only whenever $C=0$.\\
 So $$\mathcal{H}_{\alpha}(f^{\lambda})(x)=0 $$
 implies that, for all $\lambda \in S_a$: $f^{\lambda}(x)=0 $ (because $\mathcal{H}_{\alpha}$ is injective), \\
 The function $f^{\lambda}$ is entire, so we get that
 then $f^{\lambda}=0$   for all $\lambda \in \mathbb{R}$.\\
 then we have $$\int_{-\infty}^{+\infty}f(x,t)e^{-i\lambda t}dt=0$$
 thus $f(x,t)=0$ a.e.\\
This complete the proof of our main result.

\end{proof}

\end{document}